\documentclass[tbtags]{amsart}
\usepackage[T2A]{fontenc}
\usepackage[utf8]{inputenc}
\usepackage{csquotes}
\usepackage[english]{babel}
\usepackage[backend=biber,alldates=year,url=false,isbn=false,doi=false]{biblatex}
\usepackage{amsfonts}
\usepackage{amsmath}
\usepackage{amssymb}
\usepackage{mathrsfs}
\usepackage{amscd}

\usepackage{array}
\usepackage{floatflt}
\usepackage{xcolor}
\numberwithin{equation}{section}
\theoremstyle{plain}
\newtheorem{theorem}{Theorem}
\newtheorem{lemma}{Lemma}[section]
\newtheorem{propos}{Proposition}
\theoremstyle{definition}
\newtheorem{definition}{Definition}
\newtheorem{remark}{Remark}

\newcommand{\kk}{\mathbb{K}}

\newcommand{\codim}{\mathrm{codim}}

\renewcommand{\geq}{\geqslant}
\renewcommand{\leq}{\leqslant}

\begin{document}
\title{Affine toric varieties with an open orbit of an $SL_n$ action}
%\author[Н.\,Ю.~Медведь]{Никита Юрьевич Медведь}
\author{N.\,Yu.~Medved}
\address{CS Dept. at Higher School of Economics, Moscow}
\email{mednik@mccme.ru}

\keywords {Affine algebraic varieties, toric varieties, algebraic group actions, Cox ring.}

\begin{abstract}
We study affine toric varieties with an action of group $SL_n$ with a dense orbit. A characterisation in terms of $SL_n \times Q$-modules is given where $Q$ is a quasitorus. This characterisation is more explicitly expanded in case $n=3$. It is shown that in case $n=3$ the divisor class group rank is not greater than $3$, however it is unbounded when $n\geq 4$.
\end{abstract}

\maketitle

\section{Introduction}

The ground field $\kk$ is supposed to be algebraically closed of characteristic zero.

In this paper we study irreducible affine toric varieties with a regular action of the group $SL_n$ with an open orbit. A normal irreducible variety $X$ is said to be toric if it admits an effective action of an algebraic torus $T$ with an open orbit.

A classical problem in the study of algrebraic group actions is describing orbit closures, i.e. varieties with a dense orbit of a group action. For $G=SL_2$ all normal varieties with a dense $SL_2$-orbit were described by V.L.~Popov in \cite{popov_quasihomogeneous_1973} using the $U$-invariants. Unfortunately, it seems there is no such description available for $SL_n$ with $n>2$. %Of special interest is the case where the stabilizer of a generic point is trivial, such varieties are called $G$-embeddings. For $G=SL_2$ all $G$-embeddings and, more general, all normal varieties with a dense $SL_2$-orbit were described by V.L.~Popov in \cite{popov_quasihomogeneous_1973} using $U$-invariants. Unfortunately, it seems there is no such description available for $SL_n$ with $n>2$.

A well-researched class of algebraic varieties is provided by toric varieties. One of the aspects of their importance is that they provide useful examples, as many of their properties can be computed explicitly. In this paper we are interested in affine toric varieties with an open orbit of a regular action of the group $SL_n$. Their properties may be studied to gain intuition to possible properties of arbitrary affine varieties with a dense orbit of an $SL_n$ action.

In the paper \cite{gaifullin_affine_2008} there was provided a description of affine toric varieties with a dense orbit of an $SL_2$ action. 
Extending methods from that paper, we describe all irreducible affine toric varieties with an open orbit of an $SL_n$ action.

In Section \ref{Prelim} we recall basic facts about affine toric varieties and introduce the Cox construction and the total coordinate space of the variety. Also in this section we introduce the notion of a prehomogeneous vector space, that is, a vector space with a regular action of an algebraic group with an open orbit. We show that a unique prehomogeneous vector space may be associated with every affine toric variety with a dense $SL_n$-orbit. In Section \ref{Gale} we obtain the conditions for a prehomogeneous vector space to be in image of this correspondence. This allows us to reduce our problem to classification of the prehomogeneous vector spaces satisfying those conditions. For that we use the classification from \cite{sato_kimura_1977}.

In Section \ref{config} we provide a criterion for an affine space with a linear quasitorus action to be the total coordinate space of an affine toric variety. Applying this result in Section \ref{n=3} we establish which prehomogeneous vector spaces may appear in the case $n=3$. The result is contained in Theorem~\ref{thm_n_3}. One aspect of this classification is that all of them have the class group rank equal to $0, 1, 2$ or $3$. That corresponds nicely to the case $n=2$ where it follows from a well-known result due to V.L.~Popov \cite{popov_quasihomogeneous_1973} that all varieties with a dense $SL_2$ orbit, not necessarily toric ones, have the class group rank either $0$ or $1$. In Section \ref{n=4} we show that such behaviour does not continue when $n \geq 4$, in fact, for any $n \geq 4$ and any $d \in \mathbb{N}$ there is such a variety with class group rank equal $d$. Moreover, the constructed example has trivial stabilizer of the generic point.

The author would like to thank his scientific advisor S.\,A.\,Gaifullin for lots of helpful discussions. The author would also like to thank I.\,V.\,Arzhantsev for providing valuable feedback on the subject of the paper.

\section{Preliminaries}\label{Prelim}

Let $X$ be a toric variety, that is a variety with an open orbit of an effective action of an algebraic torus. It can be described in combinatorial terms by introducing a lattice $N$ and a set of stongly convex polyhedral cones in the space $N \otimes_{\mathbb{Z}} \mathbb{Q}$ satisfying several properties  that is called its fan. For details we refer the reader to \cite{cox2011toric}. We shall require only the case when $X$ is affine in which case the fan has only one cone of maximal dimension. Let its rays have $\rho_1, \ldots, \rho_\tau$ as their primitive vectors in the lattice $N$ of one-parameter subgroups. There is a bijection between those rays and prime $T$-invariant Weyl divisors, let us denote the divisor corresponding to $\rho_i$ as $D_{\rho_i}$.

The notion of the \textit{Cox ring} $\mathcal{R}(X)$ of a toric variety $X$ was formulated by D.~Cox as follows: $\mathcal{R}(X)$ is the polynomial ring of $\tau$ variables $x_1, \ldots, x_\tau$ graded by the divisor class group of $X$: a monomial $\prod x_i^{a_i}$ has degree $\left[\sum a_i D_{\rho_i}\right]$ where $[D]$ denotes class of the divisor $D$.
A similar construction can be introduced in the non-toric case, if some other conditions stay true (see, for example,~\cite{arzhantsev_cox_2015}). In the non-toric case the ring $\mathcal{R}(X)$ is not necessarily the polynomial ring. The Cox ring of a variety is unique up to a graded rings isomorphism. By $\overline{X}$ we denote $\operatorname{Spec} \mathcal{R}(X)$~--- the \textit{total coordinate space} of the variety $X$. Since the homogeneous component of $[0]$ in $\mathcal{R}(X)$ is isomophic to $\kk[X]$ we may consider the morphism $\pi: \overline{X} \to X$ which is called the \textit{Cox realization} of the variety $X$. It may be expressed as the categorical factor by the action of the characteristic quasitorus $Q=\kk[\operatorname{Cl}(X)]$.

Consider an abelian group $K$ and a $K$-graded integral $\kk$-algebra $R$. We say that a nonzero nonunit element $f$ is $K$-\textit{prime} if it is homogeneous and for any homogeneous $g, h$ such that $f|gh$ we have either $f|g$ or $f|h$. We say that $R$ is \textit{factorially $K$-graded} if every homogeneous nonzero nonunit element is a product of $K$-primes.   

Let us denote the categorical factor of $Z$ by the action of $Q$ by $Z/\!/Q$. Let us introduce a criterion for a variety $Z$ with a quasitorus action $Q$ to be a total coordinate space of an affine variety $X$:

\begin{propos}{\rm(Corollary of statement 2.3 in \cite{arzhantsev_cox_2010})}\label{universal}%see also 6.4.4 in ADHL
Consider a quasitorus $Q$ action on a normal irreducible affine variety $Z$ with an open $Q$-invariant subset $U$ such that the following conditions hold:
\begin{enumerate}
\item $\kk[Z]$ is factorially $\operatorname{Cl}(X)$-graded;
\item $\codim_Z (Z \setminus U) \geq 2$;
\item $Q$ acts freely on $U$;
\item every fiber of quotient morphism $\pi: Z \to Z/\!/Q$ intersecting $U$ consists of one $Q$-orbit.
\end{enumerate}
Let us denote $X \cong Z/\!/Q$. Then $Q$ is isomorphic to the characteristic quasitorus of $X$ and $Z$ is isomorphic to the total coordinate space $\overline{X}$.
\end{propos}

\begin{propos}\label{1}
Let $G$ be a simply connected semisimple algebraic group acting on an affine variety $X$ with an open orbit, let $Q$ be the characteristic quasitorus of the variety $X$. Then $X$ is toric if and only if there exists such a $(G \times Q)$-module $V$ with an open $(G \times Q)$-orbit that $X$ is $G$-equivariantly isomorphic to $V/\!/Q$ and $V \to V/\!/Q$ is the Cox realisation.
\end{propos}

\begin{proof}
($\Rightarrow$) Existance of the module immediately follows from applying \cite[Th. 4.2.3.2]{arzhantsev_cox_2015} to the total coordinate space of $X$. The orbit is open by \cite[Lemma~1]{arzhantsev2010homogeneous}.\\
($\Leftarrow$) The factor is obviously toric as the action of $Q$ is linear. The orbit is open again by \cite[Lemma~1]{arzhantsev2010homogeneous}.
\end{proof}

Let us now introduce the notion of a prehomogeneous vector space.
\begin{definition}
A vector space $V$ with a linear action of a connected algebraic group $G$ is called a \textit{prehomogeneous vector space} if this action has a dense orbit.
\end{definition}

In \cite{KIMURA1983} (see also \cite{kimura2002introduction}) there is a list of all prehomogeneous vector spaces for groups of type~$G_s \times Q$ where $G_s$ is simple and $Q$ is a quasitorus. From this list we obtain all prehomogeneous vector spaces of the group $SL_n \times Q$. Let $\Lambda_i$ denote the irreducible representation (Weyl module) of $SL_n$ in $\Lambda^i (\kk^n)$. We are going to identify the trivial representation $\kk^1$ with~$\Lambda_0$. It was shown in \cite{sato_kimura_1977} that any prehomogeneous vector space of~$SL_n \times Q$ must decompose into a direct sum of $\Lambda_0, \Lambda_1, \Lambda_2, \Lambda_3, S^2(\kk^n)$. Moreover,~$\Lambda_3$ may appear only if $n=6,7,8$.

The full list of possible prehomogeneous vector spaces is rather long, thus we present only the case $n=3$ and also a special case for $n>3$, which we will require later.

\begin{propos}\label{case_dim_3}
Every prehomogeneous vector space $V$ of the group $SL_3 \times Q$ is either one of the following, or the one of the conjugate to them:
\begin{enumerate}
\item $\underbrace{\Lambda_1 \oplus \ldots \oplus \Lambda_1}_l \oplus \underbrace{\Lambda_0 \oplus \ldots \oplus \Lambda_0}_r ,$ where $l=0,1,2$ and $Q$ acts with an open orbit on $\Theta(V)=\underbrace{\Lambda_0 \oplus \ldots \oplus \Lambda_0}_r$;%acts with no nonconstant rational invariants  on $\kk[V]^{SL_n}= \underbrace{\kk[\Lambda_0] \otimes \ldots \otimes \kk[\Lambda_0]}_r$;
\item $\Lambda_1 \oplus \Lambda_1 \oplus \Lambda_1 \oplus \underbrace{\Lambda_0 \oplus \ldots \oplus \Lambda_0}_r ,$ where $Q$ acts with an open orbit \\ on $\Theta(V)=\left\langle det\right\rangle \oplus \underbrace{\Lambda_0 \oplus \ldots \oplus \Lambda_0}_r$;%проблема: det все-таки функция, а те штуки пространства, не очень хорошая сумма
\item $\underbrace{\Lambda_1 \oplus \ldots \oplus \Lambda_1}_{l-1} \oplus (\Lambda_1)^{\ast} \oplus \underbrace{\Lambda_0 \oplus \ldots \oplus \Lambda_0}_r,$ where $l$ is either $2$ or $3$ and $Q$ acts with an open orbit on $\Theta(V) = \underbrace{\Lambda_0 \oplus \ldots \oplus \Lambda_0}_r \oplus \left\langle g_1 \right\rangle \oplus \ldots \oplus \left\langle g_{l-1} \right\rangle,$ where $g_i$ is the polynomial that represents the pairing between $i$-th copy of $\Lambda_1$ and~$(\Lambda_1)^{\ast}$.
\end{enumerate}
\end{propos}

\begin{propos}\label{n_by_n}
For any $n \geq 2$ the representation $\underbrace{\Lambda_1 \oplus \ldots \oplus \Lambda_1}_n \oplus \underbrace{\Lambda_0 \oplus \ldots \oplus \Lambda_0}_r $ is a prehomogeneous vector space of the group $SL_n \times Q$ whenever $Q$ acts with an open orbit on $\left\langle det\right\rangle \oplus \underbrace{\Lambda_0 \oplus \ldots \oplus \Lambda_0}_r$.
\end{propos}

\section{Gale duality and positively 2-spanning polyhedra}\label{Gale}
We provide a brief introduction to the Gale duality, for more details see, for example, \cite{matousek_lectures_2002}.

\def\mnd{d}
\def\mnn{n}

When we refer to a \textit{collection} of some objects, we mean that any element may belong to the collection in several copies, which we consider to be separate. Note that when we \textit{delete} a member of a collection, this refers to only one copy of the member, so if it existed in several copies, the others may remain. 

\begin{definition}
A \textit{point configuration} $\mathcal{A}$ in an affine space $\mathbb{A}$ over $\mathbb{Q}$ is an arbitrary collection of points $a_1, \ldots, a_\mnn$, not lying in an affine subspace of lesser dimension.

\end{definition}

\begin{remark}
It immediately follows from the definition that $\mnn \geq \mnd+1$. Further we shall consider only the case $n \geq d+2$, some of the statements may be still applicable in the case $n=d+1$ but not all of them. %We shall be interested only in the case $\mnn \geq \mnd+2$, some of the following statements may need small fixes in the case $\mnn = \mnd + 1$.

\end{remark}

\begin{definition}
A \textit{vector configuration} $\mathcal{G}$ in vector space $W$ over $\mathbb{Q}$ is an arbitrary collection of vectors $g_1, \ldots, g_\mnn$. We shall call $\mathcal{G}$ a \textit{vector configuration with zero sum} whenever $\sum g_i = 0$.

\end{definition}

Let us consider an affine space identified with $\mathbb{Q}^\mnd$ and a point configuration $\mathcal{A}$ consisting of $\mnn$ points $a_1, \ldots, a_{\mnn}$. The Gale transform of $\mathcal{A}$ is a vector configuration with zero sum $\mathcal{G}$, consisting of $\mnn$ vectors in $\mathbb{Q}^{\mnn-\mnd-1}$ defined as follows. Consider the $\mnd \times \mnn$ matrix $A$ having coordinates of the points $a_i$ as its columns. Let us denote a row vector $(\underbrace{1,1,\ldots,1}_\mnn)$ as $e$. We shall append $e$ to the matrix $A$ obtaining a $(\mnd+1) \times \mnn$ matrix $\tilde{A}$. As the points $a_i$ do not lie in an affine subspace of lesser dimension, this matrix is of rank $\mnd+1$. Now we consider the $(\mnd+1)$--dimensional subspace $W$ in $\mathbb{Q}^\mnn$, generated by its rows. Let $b_1, \ldots, b_{\mnn-\mnd-1}$ be a basis of the orthogonal subspace $W^{\perp}$. We write them into the rows of the matrix $B$. Finally, we let $g_i$ be the columns of the matrix $B$. The collection $\mathcal{G}$ of such $g_i$ is called the Gale transform of $\mathcal{A}$.

\begin{remark}
We shall note that the resulting $\mnn$ vectors do not correspond to the $\mnn$ points individually. The Gale transform only establishes a correspondence between collections. Moreover, this correspondence is not one-to-one (see the next lemmas).
\end{remark}

\begin{remark}
The Gale transform image $\mathcal{G}$ has zero sum, as the rows of $B$ are orthogonal to $e$.
\end{remark}

\begin{remark}
Addition of $e$ can be thought of as projectivization of the configuration in the affine space. A linear version of Gale duality can be defined similarly by skipping this step. Note that the Gale transform of an arbitrary vector configuration might not have the sum equal zero, as $e$ no longer necessarily lies in the linear span of the rows of $\tilde{A}$. Moreover, for the linear duality another description exists in terms of tensor products, for details refer to~\cite[Chapter~2.2.1]{arzhantsev_cox_2015}.
\end{remark}

\begin{remark}
\begin{enumerate}
\item For different choice of basis in $W^{\perp}$ the Gale transform images 
$\mathcal{G}=\{g_1,\ldots,g_\mnn\}$ and $\mathcal{G}'=\{g_1', \ldots, g_\mnn'\}$ are the same up to a linear transformation.
\item If two point collections $a$ and $a'$ are the same up to an affine transformation then the matrix $B$ is the same, thus the Gale transform images are the same.
\end{enumerate}
\end{remark}

Let us denote by $\operatorname{AffDep}(\mathcal{A})$ the set of all linear dependancies of the matrix $\tilde{A}$ columns $\{\alpha \in \mathbb{Q}^\mnn: \alpha_1 a_1 + \ldots + \alpha_\mnn a_\mnn =0, \alpha_1+ \ldots + \alpha_\mnn = 0\}$.
As $\operatorname{AffVal}(\mathcal{A})$ we denote the set $\{(f(a_1), \ldots, f (a_\mnn)) \mid f: \mathbb{Q}^\mnd \to \mathbb{Q} \text{~--- an affine function} \}$.
Similarly, $\operatorname{LinDep}(\mathcal{G})$ stands for $\{\alpha \in \mathbb{Q}^{\mnn-\mnd-1}: \alpha_1 g_1 + \ldots + \alpha_\mnn g_{\mnn} =0\}$ and $\operatorname{LinVal}(\mathcal{G})$ for $\{(f(g_1), \ldots, f (g_\mnn)) \mid f: \mathbb{Q}^\mnd \to \mathbb{Q} \text{~--- linear function} \}$.

\begin{lemma}\label{dual}
Let $\mathcal{A}$ be a point configuration in space $\mathbb{Q}^\mnd$ and $\mathcal{G}$~--- the Gale transform of $\mathcal{A}$. Then $\operatorname{AffVal}(\mathcal{A})=\operatorname{LinDep}(\mathcal{G}), \operatorname{AffDep}(\mathcal{A})=\operatorname{LinVal}(\mathcal{G})$.
\end{lemma}
\begin{proof}

For any $(b_1, \ldots, b_\mnn) \in \operatorname{AffVal}(\mathcal{A})$ there exists an affine function $f: \mathbb{Q}^\mnd \to \mathbb{Q}$ such that $f(a_i)=b_i$ for all $i$. Let us consider the coordinates of $f$ by denoting its linear part as a row vector $f_1$ and its constant as $f_0$. As the  rows of $\tilde{A}$ and $B$ are othogonal, we get matrix equations $AB^T = 0, eB^T = 0$. By applying the equation $b = f_1 A + f_0 e$ we get the following:
$$\sum b_i g_i = bB^T = (f_1A+f_0e)B^T = f_1AB^T+f_0eB^T = 0.$$
To prove the reverse note that if $\sum b_i g_i = 0,$ then $b$ is orthogonal to all the rows of the matrix $B$, thus equal to a linear combination of the rows of $\tilde{A}$.

The proof of the second equality is essentially the same.\end{proof}

\begin{lemma}\label{general_position}
All the points of $\mathcal{A}$ lie in general position if and only if every $\mnn-\mnd-1$ vector of $\mathcal{G}$ form a basis in $\mathbb{Q}^{\mnn-\mnd-1}$.
\end{lemma}
\begin{proof}
Suppose some $\mnn-\mnd-1$ vectors of $\mathcal{G}$ do not form a basis of $\mathbb{Q}^{\mnn-\mnd-1}$. Then there is a linear dependency $\sum b_i g_i = 0$ with no more than $\mnn-\mnd-1$ nonzero elements. By previous Lemma $b$ may be interpreted as an element of $\operatorname{AffVal}(\mathcal{A})$, which means that there exists such an affine function $f$ that it is zero at at least $\mnd+1$ point from $a$. This means there is a hyperplane containing at least $\mnd+1$ point.
Those arguments may be easily followed backwards to prove the reverse.
\end{proof}

\begin{lemma}\label{in_one_face}
Let $I$ be a subset of $\{1,\ldots,\mnn\}$, then the points $\{a_i \mid i \in I\}$ lie in one face of the convex hull $\operatorname{conv}(\mathcal{A})$ if and only if $0 \in \operatorname{conv}\{g_j \mid j \not\in I\}$.
\end{lemma}
\begin{proof}
If points $\{a_i \mid i\in I\}$ lie in a common face of $\operatorname{conv}(\mathcal{A})$ then there is an affine function $f$ which is zero at $a_i$ and nonnegative at all other points of configuration. Moreover, at some point of configuration it is positive, as all the points of $\mathcal{A}$ may not lie in a hyperplane.
Let us consider the vector $b = \left(f(a_1),f(a_2), \ldots, f(a_\mnn)\right)$, it belongs to $\operatorname{LinDep}(\mathcal{G})$, which implies $\sum b_j g_j = 0$. As $b_i = 0$ for $i \in I$, we can delete those terms from the sum, considering $\sum\limits_{j \not\in I} b_j g_j = 0$ where all $b_j$ are nonnegative. This equation exactly means that $0 \in \operatorname{conv}\{g_j \mid j \not\in I\}$.

It is easy to see that all the logical transitions are reversible.
\end{proof}

\section{Criterion of a Cox ring of an affine toric variety}\label{config}

For the next definition and its discussion in (unrelated) context of combinatorial geometry refer to \cite{2013arXiv1304.7186P}.

\begin{definition}
We shall call a vector configuration in $W$ \textit{positively 2-spanning} if for any linear hyperplane $H$ both open halfspaces $H^+$ and $H^-$ contain at least 2 vectors from the configuration, not necessarily different. An equivalent definition is that when we delete any vector from the configuration, the remaining do not all lie in one closed halfspace.
\end{definition}

\begin{lemma}\label{dual_to_polygon}
A vector configuration $g_1 \ldots g_\mnn$ spanning $\mathbb{Q}^\mnd$ is positively 2-spanning if and only if the Gale dual point configuration $a_1 \ldots a_\mnn$ in $\mathbb{Q}^{\mnn-\mnd-1}$ is the set of vertices of a convex polyhedron without repetitions.
\end{lemma}
\begin{proof}
Let configuration $\mathcal{G}$ be not positively 2-spanning. It means that there exists a vector $g_j$ such that all other lie in a closed halfspace, which means there exists a linear function $h$ such that $h(g_i) \geq 0$ for all $i \neq j$. Let $b_i$ be $h(g_i)$. As the configuration spans the entire space, $h$ cannot be zero on all the vectors from the configuration, thus $\sum\limits_{i\neq j} b_i > 0$, which implies $b_j<0$. By Lemma~\ref{dual} we obtain $b \in \operatorname{AffDep}{\mathcal{A}}$, which implies $\sum b_i a_i = 0$. Dividing by $|b_j|$ and isolating $a_j$ on one side of the equation we obtain $\sum\limits_{i \neq j} \frac{b_i}{|b_j|} a_i = a_j$. From $\sum g_i = 0$ it immediately follows that $\sum b_i = 0$, which implies  $\sum\limits_{i\neq j} \frac{b_i}{|b_j|} = 1$. Thus existance of such a vector $b\in \operatorname{AffDep}{\mathcal{A}}$ with one negative coordinate $b_j$ and all other nonnegative is equivalent to existance of a vertex $a_j$ that lies in convex hull of the others. This condition means exactly that the configuration $a_i$ is not a set of vertices of a convex polyhedron.
\end{proof}

\begin{lemma}\label{scale_one_vector}
Suppose a configuration $g_1 \ldots g_\mnn$ is positively 2-spanning. Then for any positive rational $\lambda_i$ the configuration $\{\lambda_i g_i \}$ is positively 2-spanning too.
\end{lemma}
\begin{proof}
Obviously if a vector $w_i$ is in some open halfspace, then the vector $\lambda_i w_i$ lies in the same halfspace. Thus the lemma is trivial.
\end{proof}

\begin{lemma}\label{monom}
Suppose a vector configuration $g_1, \ldots, g_s$ such that its convex hull contains some neighbourhood of $0$. Fix some index $k$. Then there exists a nonnegative linear dependance $\sum \alpha_i g_i$ with $\alpha_k > 0$.
\end{lemma}
\begin{proof}
If we triangulate the surface of the convex hull we obtain a partition of the space into strongly convex simplicial cones. Consider the cone containing $-g_k$. Then by a well-known lemma there is a coefficient $N$ such that $-Ng_k$ may be expressed as a nonnegative linear combination of the elements that are generating the rays of this cone.
\end{proof}

Let $\overline{M}$ be a finitely generated abelian group.
We set $M$ to be the factorgroup $\overline{M}/\operatorname{Tor}(\overline{M})$.
For an element $\overline{w}$ of $\overline{M}$ let $w$ denote its image in $M$.
Let $M_{\mathbb{Q}} = M \otimes_{\mathbb{Z}} \mathbb{Q}$.
\begin{definition}
In the above notation we say that a collection
$\{\overline{w_i}\}$ in $\overline{M}$ satisfies the condition~$\ast$ if two following conditions hold:
\begin{enumerate} %ХОРОШО БЫ СДЕЛАТЬ РИмСКИМИ
\item the configuration $\{w_i\}$ in $M_{\mathbb{Q}}$ is positively 2-spanning;
\item if we delete any $\overline{w_i}$, the rest generate $\overline{M}$.
\end{enumerate}
\end{definition}

If $Q$ is a quasitorus let $\mathfrak{X}(Q)$ denote the character group of $Q$, that is, the group of all homomorphisms $Q \to \kk^{\times}$.

Let us prove the following criterion, which allows to determine whether an affine space with a quasitorus action is a total coordinate space of an affine toric variety:

\begin{propos}\label{crit}
Suppose a quasitorus $Q$ acts linearly on $V = \kk^l$. Let $\overline{w_i} \in \overline{M}=\mathfrak{X}(Q)$ be the weights of the coordinate functions $x_i$. Then $V$ equipped with this action of $Q$ is the Cox realisation of an affine toric variety $V/\!/Q$ if and only if $\{\overline{w_i}\}$ satisfy condition~$\ast$.
\end{propos}
\begin{proof}
($\Rightarrow$) Let $\pi: V \stackrel{/\!/Q}{\twoheadrightarrow} X$ be the Cox representation of a toric variety $X \cong V/\!/Q$. Let us prove that the condition~$\ast$ holds. By Proposition \ref{universal} there is an open subset $U \subset V$, such that:
\begin{itemize}
\item $\codim_{V} V\setminus U \geq 2$, 
\item $Q$ acts freely on $U$,
\item every fiber of $\pi$ intersecting $U$ consists of exactly one $Q$-orbit.
\end{itemize}

Suppose there is an index $j$ such that $\{\overline{w_1}, \ldots \overline{w_{j-1}}, \overline{w_{j+1}}, \ldots \overline{w_l}\}$ do not generate $\overline{M}$. We can rephrase it as follows: let
$$B=\left<\overline{w_1}, \ldots \overline{w_{j-1}}, \overline{w_{j+1}}, \ldots \overline{w_l}\right> \subset \overline{M},$$ then $B$ is a proper subgroup. 

\begin{lemma}\label{proper_subgroup}
Assume $\overline{M} = \mathfrak{X}(Q)$. Then a subgroup $B\subset \overline{M}$ is proper if and only if there is an element $s\in Q$ such that $s$ is not the identity and $b(s)=1$ for all $b\in B$.
\end{lemma}
\begin{proof}
Let $F$ be a finitely generated free abelian group, $\tau: F \to \overline{M}$ a surjection. Let $C=\tau^{-1}(B)$. From $B \varsubsetneq \overline{M}$ immediately follows $C \varsubsetneq F$. Let us choose coordinated bases $\{f_i\}$ and $\{d_i f_i\}$ in free abelian groups $F$ and $C$. If $d_k=0$ for some index $k$ then consider $\phi: F \to \kk^{\times}$ given by formula $\phi(f_k) = \zeta$ and $\phi(f_i) = 1$ for all $i \neq k$ where $\zeta$ is an arbitrary nonunit element of $\kk^{\times}$. If all $d_i \neq 0$ then pick an arbirary $k$ and set $\phi(f_k) = \sqrt[d_k]{1}, \varphi(f_i) = 1$ for all $i \neq k$. See that in both cases $\phi \neq 1$, but $\phi(\ker \tau)=\phi(C) = 1$. Thus $\phi$ factors through $\tau$ and we obtain some $s: \overline{M} \to \kk^{\times}$ by formula $s = \phi(\tau^{-1}).$ This is the required $s$. The second part of the statement is obvious.
\end{proof}

Let us consider the element $s \in Q$ from Lemma \ref{proper_subgroup}. Let us show that $Q$ does not act freely on $U$. Indeed, $\overline{w_i}(s) = 1$ for all $i \neq j$, which means that all coordinates $x_i$ are invariant with respect to $s$ for $i \neq j$. Thus $s$ acts trivially on $U \cap \{x_j=0\}$. But this set is nonempty since $\codim_{V} V\setminus U \geq 2$.

Now let us assume that the configuration $\{w_i\}$ is not positively 2-spanning. It means there is an index $j$ and a closed halfspace $\alpha^+$ such that for all $i \neq j$ we have $w_i \in \alpha^+$. By $\alpha$ let us denote the hyperplane that is the border of $\alpha^+$. Let $K \subset \{1,\ldots, l\} \setminus\{j\}$ be the set of indices such that for all $k \in K$ we have $w_k \not \in \alpha$.

If $K$ is nonempty then let us pick an arbitrary $k \in K$ and consider the set $U_k=(\{x_j=0\} \cap U) \cap \{x_k \neq 0\}$. It is not empty as it is an intersection of nonempty open sets in $\{x_j = 0\}$. Thus we may consider an arbitrary vector $v \in U_k$. Let us consider another vector $v'$ obtained from $v$ by replacing the $k$-th coordinate by $0$. 
Every regular $Q$-invariant containing $x_k$ should contain $x_j$, thus equals $0$ at both $v$ and $v'$. The regular $Q$-invariants that do not contain $x_k$ are equal on $v$ and $v'$. Thus all regular $Q$-invariants are the same on $v$ and $v'$ but they lie in different orbits of the quasitorus, which contradicts the assumptions.

If $K$ is empty then there is no regular $Q$-invariant containing $x_j$. Analogously to the previous paragraph we may consider $v \in U \setminus \{x_j = 0\}$ and $v'$ obtained by replacing the $k$-th coordinate by $0$.
The same arguments apply.

($\Leftarrow$)
Let us assume that the condition $\ast$ hold. We again apply Proposition~\ref{universal}. The factoriality holds automatically as the polynomial ring is factorial, thus it is homogeneously factorial with respect to any grading. As $U$ let us choose the set
$$U = V \setminus \bigcup\limits_{i\neq k} \{ x_i = 0, x_k = 0\},$$
that is the set of points that have no more than one zero coordinate. We see that $\codim_V V\setminus U = 2$ holds.  

Let us show that $Q$ acts freely on $U$. Indeed, suppose there is an element $s\in Q$ stabilizing a point $u\in U$. That means that $s$ acts trivially on all coordinate functions not vanishing at $u$. So there is $j$ such that %except, perhaps, one (that turns to zero at $u$). So let us assume that
$s$ acts trivially on all coordinate functions except $x_j$.
As $\{\overline{w_i} | i \neq j\}$ generate $\overline{M}$ we see that for every $\overline{w} \in \overline{M}$ there is a representation $\overline{w} = \sum\limits_{i\neq j} \alpha_i \overline{w_i}$. All $\overline{w_i}(s)=1$ for $i\neq j$ thus $\overline{w}(s)=1$ for all $\overline{w} \in \overline{M}$.
Thus $s$ equals $1$.

Finally, let us show that for every $u \in U$ the preimage $\pi^{-1}(\pi(u))$ is exactly one $Q$-orbit. Let us fix some point $u \in U$ and suppose that every its coordinate except $x_j$ is not $0$. Let $u'$ be another point in $\pi^{-1}(\pi(u))$, we are going to show that every its coordinate except $x_j$ is also not $0$. This would imply that $u' \in U$. Let us pick a coordinate $x_k$ where $k \neq j$. As the configuration $\{w_i\}$ is positively 2-spanning there is a nonnegative linear dependance $\sum\limits_{i \neq j} \alpha_i w_i = 0$. By Lemma~\ref{monom} one can get the coefficient $\alpha_k$ to be positive. Consider a corresponding $M$-homogeneous monomial $m = \prod\limits_{i \neq j} x_i^{\alpha_i}$. We raize it to some power $d$ so that $m^d$ is $\overline{M}$-homogeneous, that is, $Q$-invariant. Note that $m^d$ is nonzero at $u$, thus it is nonzero at every point in $\pi^{-1}(\pi(u))$. This implies that for any $u'$ in $\pi^{-1}(\pi(u))$ its $k$-th coordinate is nonzero. As this stands for every $k$ except $j$, we obtain that $u' \in U$, in other words $\pi^{-1}(\pi(u)) \subset U$. All the $Q$-orbits in $U$ are of the same dimension, thus one of them cannot lie in the closure of another, so we obtain that $\pi^{-1}(\pi(u))$ consists of only one $Q$-orbit.
\end{proof}

From Prop.~\ref{1} and Prop.~\ref{crit} obviously follows the following proposition:
\begin{propos}\label{result}
Affine toric varieties with an action of a simply connected semisimple group $G$ with an open orbit are categorical factors by the action of $Q$ of $(G\times Q)$-modules with an open $G\times Q$-orbit for which the condition~$\ast$ holds .
\end{propos}

We also provide two lemmas about positively 2-spanning configurations that are going to be used later.

\begin{lemma}\label{d+3}
Suppose a collection $\{w_1,\ldots w_n\}$ in an $s$-dimensional space is positively 2-spanned. Then either $s=0$ or $n \geq s+3$.
\end{lemma}
\begin{proof}
Suppose $s \neq 0$ and consider a hyperplane spanned by $s-1$ vectors in the collection. Then by definition there are at least 2 vectors on both sides of this hyperplane thus there are at least $s+3$ elements in the collection.
\end{proof}

\begin{lemma}\label{projection}
Suppose $w_1, \ldots w_s \in M, M=A\oplus B$. %and $A$ is the subspace generated by $\{w_i \mid i \in I \}$ such that all of them appear more than once in the collection (but not necessarily all such elements). 
If the collection $\{w_i\}$ is positively 2-spanning then the collection of their projections $\{proj_{B} (w_i)\}$ onto the second summand is also positively 2-spanning.
\end{lemma}
The proof immediately follows from the definition.

\section{Affine toric varieties with an action of the group $SL_3$ with an open orbit}\label{n=3}

\begin{propos}\label{prop_n_3}
a) Every $(SL_3\times Q)$-module $V$ with an open $SL_3\times Q$-orbit for which the condition~$\ast$ holds is either one of the following, or one of the conjugate to them:
\begin{enumerate}
\item \begin{enumerate}
\item $\{0\}$, where $\dim Q = 0$;
\item $\Lambda_1$, where $\dim Q = 0$;
\item $\Lambda_1 \oplus \Lambda_1$, where $\dim Q = 0$;
\item $\Lambda_1 \oplus \Lambda_1 \oplus \underbrace{\Lambda_0 \oplus \ldots \oplus \Lambda_0}_r ,$ where $r=0,1$ and $\dim Q = 1$;
\item $\Lambda_1 \oplus \Lambda_1 \oplus \underbrace{\Lambda_0 \oplus \ldots \oplus \Lambda_0}_r ,$ where $r=2$ and $\dim Q = 2$;
\end{enumerate}
\item 
\begin{enumerate}
%\item $\Lambda_1 \oplus \Lambda_1 \oplus \Lambda_1 \oplus \underbrace{\Lambda_0 \oplus \ldots \oplus \Lambda_0}_r ,$ where $r=0$ and $\dim Q = 0$;
\item $\Lambda_1 \oplus \Lambda_1 \oplus \Lambda_1 \oplus \underbrace{\Lambda_0 \oplus \ldots \oplus \Lambda_0}_r ,$ where $r=0$ and $\dim Q = 1$;
\item $\Lambda_1 \oplus \Lambda_1 \oplus \Lambda_1 \oplus \underbrace{\Lambda_0 \oplus \ldots \oplus \Lambda_0}_r ,$ where $r=0,1$ and $\dim Q = 2$;
\item $\Lambda_1 \oplus \Lambda_1 \oplus \Lambda_1 \oplus \underbrace{\Lambda_0 \oplus \ldots \oplus \Lambda_0}_r ,$ where $r=2$ and $\dim Q = 3$;
\end{enumerate}
\item $\underbrace{\Lambda_1 \oplus \ldots \oplus \Lambda_1}_{l-1} \oplus (\Lambda_1)^{\ast},$ where $l=2,3$ and $\dim Q = l-1$,
\end{enumerate}
where we assume that $Q$ acts with an open orbit on $\Theta(V)$ as in Prop.~\ref{case_dim_3}.\\
b) Every of the listed cases exists, that is, there is such a group $\overline{M}$ and a set of weights that $Q$ acts with an open orbit on $\Theta(V)$ and the condition~$\ast$ holds.
\end{propos}

\begin{proof}
Let $d$ denote $\dim Q$.

Let us start with the first case from Prop.~\ref{case_dim_3}: let $$V=\underbrace{\Lambda_1 \oplus \ldots \oplus \Lambda_1}_l \oplus \underbrace{\Lambda_0 \oplus \ldots \oplus \Lambda_0}_r,$$ where $l=0,1,2$ and $Q$ acts with an open orbit on $\underbrace{\Lambda_0 \oplus \ldots \oplus \Lambda_0}_r$. Let us skip the case $l=0$ which immediately comes down to case (1).

Let us denote the $Q$-weights on the summands as $\overline{v_1}, \ldots, \overline{v_l}, \overline{w_1}, \ldots, \overline{w_r}$ where the first $l$ correspond to $\Lambda_1$ and the next $r$ correspond to $\Lambda_0$. Thus we want to find inequalities for $r$ by checking the condition~$\ast$ for the collection $$\overline{v_1}, \overline{v_1}, \overline{v_1}, \overline{v_2}, \overline{v_2}, \overline{v_2}, \ldots, \overline{v_l}, \overline{v_l}, \overline{v_l}, \overline{w_1}, \overline{w_2}, \ldots ,\overline{w_r}.$$

Now let us consider the subspace $A$ generated by $v_i$ and its dimension $a$. Then after applying Lemma~\ref{projection} we obtain a $(d-a)$-dimensional space with no more than $r$ nonzero projections that come from $w_1, \ldots, w_r$. By applying Lemma~\ref{d+3} we see that either $d-a=0$ or $r \geq d-a+3$. Suppose the second case holds. As $Q$ acts with an open orbit on $\underbrace{\Lambda_0 \oplus \ldots \oplus \Lambda_0}_r$, it immediately follows that $r \leq d$. Combining the inequalities we see that $d-a+3 \leq d$, which means that $a \geq 3$. On the other hand, $a \leq l \leq 2$. Thus $r \geq d-a+3$ is impossible, so only the case $d-a=0$ remains. 

\begin{lemma}\label{max}
If $l \geq d$ the following inequality holds: $l + \frac{r}{2} \geq d+1$.
\end{lemma}
\begin{proof}
As $l \geq d$ we may pick $d-1$ element $v_1, \ldots, v_{d-1}$ and consider a linear hyperplane $\alpha$ through them. On each side there are at least 2 elements, thus there is either 1 element of the remaining $v_d, \ldots, v_{l-(d-1)}$ or 2 elements of $w_1, \ldots, w_r$. Taking into account both sides we obtain $2(l-(d-1))+r \geq 4$. This inequality is equivalent to the one claimed.
\end{proof}

If $d=a=0$ then $r$ is also $0$ and we are in the case (2) or (3). If $d=a=1$ then $r \leq d$ thus $r=0,1$. By Lemma~\ref{max} we get $l \geq d+1 = 2$, thus $l=2$. We obtained the case (4). Finally, if $d=a=2$ then as $l \geq a$ we have $l=2$. By Lemma~\ref{max} we get $2+\frac{r}{2} \geq 2+1$ thus $r \geq 2$. As $r \leq d = 2$ we have $r=2$. This is the case (5).

Now let us consider the second case from the Prop.~\ref{case_dim_3}: let $$V = \Lambda_1 \oplus \Lambda_1 \oplus \Lambda_1 \oplus \underbrace{\Lambda_0 \oplus \ldots \oplus \Lambda_0}_r ,$$ where $Q$ acts with an open orbit \\ on $\left\langle det\right\rangle \oplus \underbrace{\Lambda_0 \oplus \ldots \oplus \Lambda_0}_r$. We keep the notation from the previous case. By applying Lemma~\ref{projection} we obtain a $(d-a)$-dimensional space with no more than $r$ nonzero projections that come from $w_1, \ldots, w_r$. By applying Lemma~\ref{d+3} we see that either $d-a=0$ or $r \geq d-a+3$. Suppose the second case holds. As $Q$ acts with an open orbit on $\left\langle det\right\rangle \oplus \underbrace{\Lambda_0 \oplus \ldots \oplus \Lambda_0}_r$, we obtain $d \geq r+1$. Combining the inequalities we see that $d-a+3 \leq d-1$, which means that $a \geq 4$. On the other hand, $a \leq l = 3$, thus $r \geq d-a+3$ is impossible, so only the case $d-a=0$ remains.

By Lemma~\ref{max} with $l=3$ we get $d \leq 3$. As $d \geq r+1$ we know that $d \geq 1$. If $d=a=1$ then as $r+1 \leq d$ we obtain $r=0$. If $d=a=2$ by Lemma~\ref{max} we get $3+\frac{r}{2} \geq 2+1$ thus it does not add any restrictions. As $r+1 \leq d = 2$ we have $r \leq 1$. Finally, if $d=a=3$ by Lemma~\ref{max} we get $3+\frac{r}{2} \geq 3+1$ thus $r \geq 2$. On the other hand we know that $r+1 \leq d = 3$. Thus $r=2$.

Now let us consider the last case from Prop.~\ref{case_dim_3}: let $$V=\underbrace{\Lambda_1 \oplus \ldots \oplus \Lambda_1}_{l-1} \oplus (\Lambda_1)^{\ast} \oplus \underbrace{\Lambda_0 \oplus \ldots \oplus \Lambda_0}_r,$$ where $l$ is either $2$ or $3$ and $Q$ acts with an open orbit on $\underbrace{\Lambda_0 \oplus \ldots \oplus \Lambda_0}_r \oplus \left\langle g_1 \right\rangle \oplus \ldots \oplus \left\langle g_{l-1} \right\rangle,$ where $g_i$ is the polynomial that represents the coupling between $i$-th copy of $\Lambda_1$ and $(\Lambda_1)^{\ast}.$ Let us again keep the notation from the previous cases. By applying Lemma~\ref{projection} we obtain a $(d-a)$-dimensional space with no more than $r$ nonzero projections that come from $w_1, \ldots, w_r$. By applying Lemma~\ref{d+3} we see that either $d-a=0$ or $r \geq d-a+3$. Suppose the second case holds. As $Q$ acts with an open orbit on $\underbrace{\Lambda_0 \oplus \ldots \oplus \Lambda_0}_r \oplus \left\langle g_1 \right\rangle \oplus \ldots \oplus \left\langle g_{l-1} \right\rangle,$ we obtain $d \geq r+(l-1)$. Combining the inequalities we see that $d-a+3 \leq d-l+1$, which means that $a \geq l+2$. On the other hand, $a \leq l$, thus $r \geq d-a+3$ is impossible, so only the case $d-a=0$ remains.

As $d \geq r+(l-1) \geq r+1$ we know that $d \geq 1$. If $a=d=1$ then as $d\geq r+1$ we obtain $r=0$. Also as $d \geq l-1$ we see that $l=3$ is impossible in this casem thus $l=2$. If $a=d=2$ then from $d \geq r+1$ we obtain $r \leq 1$, by Lemma~\ref{max} we get $l+\frac{r}{2} \geq 2 + 1 = 3$, thus $l = 3$. As $d \geq r+(l-1) = r+2$ we now get $r=0$.

This ends the proof of part a) as we went through all cases of Prop.~\ref{case_dim_3}. Now to prove b) we present a set of weights in $M=\mathbb{Z}^{d}$ satisfying the constraints. 

\begin{tabular}{cccccc}
Case & $v_1$ & $v_2$ & $v_3$ & $w_1$ & $w_2$\\
\hline
1a & - & - & - & - & - \\
1b & $0$ & - & - & - & - \\
1c & $0$ & $0$ & - & - & - \\
1d & $1$ & $-1$ & - & $1$ if $r=1$ & - \\
1e & $(1,0)$ & $(0,1)$ & - & $(-1,-1)$ & $(-1,-2)$ \\
2a & $1$ & $1$ & $-1$ & - & - \\
2b & $(1,0)$ & $(0,1)$ & $(-1,-2)$ & $(1,0)$ if $r=1$ & - \\
2c & $(1,0,0)$ & $(0,1,0)$ & $(0,0,1)$ & $(-1,-1,-2)$ & $(-1,-2,-1)$ \\
3,$l=2$ & $1$ & $-1$ & - & - & - \\
3,$l=3$ & $(1,0)$ & $(0,1)$ & $(-1,-1)$ & - & - \\
\end{tabular}

%For cases (1a--c) the statement is obvious. In the case (1d) we may take the weights $v_1=1,v_2=-1,w_1=1$, where we define $w_1$ only if $r=1$. In the case (1e) we consider the weights $v_1 = (1,0), v_2 = (0,1), w_1 = (-1,-1), w_2 = (-1,-2)$.

%For case (2a) consider the weights $v_1=1,$ 
\end{proof}

The following theorem immediately follows from Proposition~\ref{prop_n_3} and from Proposition~\ref{result}.

\begin{theorem}\label{thm_n_3}
Every affine toric variety $X$ with an action of the group $SL_3$ with an open orbit is the categorical factor of a module $V(X)$ from  Prop.~\ref{prop_n_3}. Moreover, if $X\underset{SL_3}{\cong}Y$ then $V(X)\underset{SL_3 \times Q}{\cong} V(Y)$.
\end{theorem}

\section{A series of examples with arbitrary large class group rank}\label{n=4}
As we have discussed in the introduction, in the case $n=2$ there is a result due to V.\,L.\,Popov that the variety is either a homogeneous space or there is exactly one divisor outside of the open orbit. This means that the divisor class group in this case has rank 1, in fact, it is shown that is equal to $\mathbb{Z}\oplus\mathbb{Z}_m$. Our result in the previous section shows that in the case $n=3$ at least for the toric varieties the dimension of the characteristic quasitorus is $3$ or less, which means that the rank of the class group is $3$ or less. Now we show in the following theorem that such behaviour does not continue for $n>3$.

\begin{theorem}\label{thm_n_4}
For every $n\geq 4, d \geq 1$ there exists a $SL_n$-embedding with the class group rank $d$.
\end{theorem}
\begin{proof}

Let us consider the module $\underbrace{\kk^n \oplus \dots \oplus \kk^n}_n \oplus \underbrace{\kk \oplus \dots \oplus \kk}_{d-1}$. Let $x_{ij}$ denote the coordinates in $i$-th copy of the module $\kk^n$ and $y_i$ denote the coordinate in $i$-th copy of $\kk$. We equip the module with a $d$-dimensional torus $T$ acting with weights $v_1, \ldots, v_n, w_1, \ldots, w_{d-1}$. We want to construct such weights that the condition~$\ast$ holds and there is an open orbit of $SL_n \times Q$.

Let us set all the weights $v_i$ for $i \geq 5$ to be equal $0$. Now consider a convex polygon with $d+3$ vertices in $\mathbb{Q}^2$ and let the weights $v_1, v_2, v_3, v_4, w_1, \ldots, w_{d-1}$ be equal to the Gale transform image of the collection of its vertices. They generate some lattice, isomorphic to $\mathbb{Z}^d$, let us denote it as $W$. Lemma~\ref{dual_to_polygon} ensures that the configuration $\{v_1,v_2,v_3,v_4,w_1,\ldots,w_{d-1}\}$ is positively 2-spanning.

For the orbit $SL_n \times Q$ to be open we need the weights of the independent generators of the algebra of $SL_n$-invariant, that is $y_1, \ldots, y_{d-1}, \det (x_{ij})$, to be linearly independent, as Proposition~\ref{n_by_n} tells us. Those weights are $w_1, \ldots, w_{d-1}, nv_1+nv_2+\ldots+nv_n$. However by construction we have $v_1+v_2+v_3+v_4 = -(w_1+\ldots+w_{d-1})$ and all other $v_i$ are equal $0$. Thus we have to tinker with our set of weights as follows. We multiply all the weights except of $v_1$ by a factor of $2$, denoting the new weights as $v_i'$ and $w_j'$. By Lemma~\ref{scale_one_vector} the result is still positively 2-spanning. By Lemma~\ref{general_position} the vectors $v_1, w_1, w_2, \ldots, w_{d-1}$ are linearly independent, thus the vector
$$\chi_{det}=n(v_1'+v_2'+\ldots+v_n')=n(v_1+2v_2+\ldots+2v_n)=$$
$$=-2n(w_1+\ldots+w_{d-1})-nv_1=-n(w_1'+\ldots+w_{d-1}')-nv_1'$$
is linearly independent with the system $w_1', w_2', \ldots, w_{d-1}'$.

Let us denote the lattice generated by new weights as $W'$. Now we need to check the condition~$\ast$ for the new set of weights. By Lemma~\ref{scale_one_vector} the new configuration is positively 2-spanning. It remains to establish that if we delete any element, the remaining do generate $W'$. As the configuration contains multiple copies of $v_i'$, their deletion cannot change the generated lattice. Thus we consider only the case of deleting some $w_j'$. But $w_j' = - (v_1'+v_2'+\ldots+v_n')-(w_1'+\ldots+w_{j-1}+w_{j+1}+\ldots+w_{d-1}')$, which means it can be obtained from the other weights, which implies the lattice remains the same. This concludes our proof, as the obtained configuration satisfies condition~$\ast$ and provides an open $SL_n \times Q$-orbit.

\end{proof}

%\bibliographystyle{plain}
%\bibliography{medved}
\printbibliography

\end{document}